\documentclass{amsart}

\usepackage{amsmath, amsbsy,amsfonts,amsopn,amstext, euscript}
\usepackage{amscd, epic, amssymb}

\usepackage{amsxtra}

\usepackage[dvips]{rotating}
\usepackage{portland}

\usepackage{graphicx}

\usepackage[latin1]{inputenc}

\usepackage{color,epsfig}
\usepackage{times}
\usepackage[T1]{fontenc}
\usepackage{latexsym}
\usepackage{amsfonts}
\usepackage{amsmath}
\usepackage{amssymb}

\newtheorem{thm}[equation]{Theorem}

\newtheorem{lem}[equation]{Lemma}
\newtheorem{prop}[equation]{Proposition}
\newtheorem{defn}[equation]{Definition}

\newcommand{\N}{{\mathbb{N}}}
\newcommand{\Z}{{\mathbb{Z}}}
\newcommand{\Q}{{\mathbb{Q}}}

\newcommand{\C}{{\mathbb{C}}}
\newcommand{\F}{{\mathbb{F}}}

\newcommand{\MMM}{{\mathcal{M}}}
\newcommand{\HHH}{{\mathcal{H}}}

\def\Aut{{\rm Aut}}
\def\GL{{\rm GL}}
\def\PGL{{\rm PGL}}

\def\AGL {{\rm AGL}}
\def\ASL {{\rm ASL}}

\def\Mon{{\rm Mon}}
\def\ind{{\rm ind}}

\def\a{{\alpha}}

\def\s{{\sigma}}
\def\t{{\tau}}

\def\ut{{\underline{\tau}}}
\def\us{{\underline{\sigma}}}

\title{Generating sets of Affine groups of low genus}

\keywords{affine groups, Hurwitz spaces,  covers of curves}

\subjclass[2000]{Primary: 14H10, 14H30,  Secondary: 14H45, 14H35}

\author{K. Magaard, S.Shpectorov, G. Wang}
\address{School of Mathematics, University of Birmingham,
  Edgbaston, Birmingham B15 2TT, U.K} 

\email{K.M.: k.magaard@bham.ac.uk}
\email{S.S..: s.shpectorov@bham.ac.uk}
\email{G.W..: wangg@maths.bham.ac.uk}
\begin{document}

\begin{abstract}
We describe a new algorithm for computing braid orbits on Nielsen classes. As an application 
we classify all families of affine genus zero systems; that is all families of coverings of the Riemann sphere by 
itself such that the monodromy group is a primitive affine permutation group.  
\end{abstract}

\maketitle

\section{Introduction}
Let $G$ be a finite group. By a $G$-curve we mean a compact, connected Riemann surface $X$ of genus $g$ such that 
$G \leq \Aut(X)$. By a $G$-cover we mean the natural projection $\pi$ of $G$-curve $X$ to its orbifold $X/G$. 
In our situation $X/G$ is a Riemann surface of genus $g_0$ and $\pi$ is a branched cover. 
We are interested in Hurwitz spaces which are moduli spaces of $G$-covers. 
By $\HHH_r^{in}(G,g_0)$ we 
mean the Hurwitz space of equivalance classes of $G$-covers which are brached over $r$ points, such that
$g(X/G) = g_0$. We are mostly interested in the case $g_0 =0$ in which case we will simply write 
$\HHH_r^{in}(G)$.

Hurwitz spaces are used to study the moduli space ${\mathcal M}_g$ of curves of genus $g$.
For example Hurwitz himself showed the connectedness of ${\mathcal M}_g$ by first showing that every curve admits 
a simple cover onto $P^1\C$ and then showing that the Hurwitz space of simple covers is connected. 

The study of Hurwitz spaces is also closely related to the inverse problem of Galois theory. 
The precise connection was given by Fried and V\"olklein in \cite{FV}.

\begin{thm}[Fried-V\"olklein]  The following are true:
\begin{enumerate}
\item $\HHH_r^{in}(G)$ is an affine algebraic set which is defined over $\Q$.
\item If $G$ is a group with $Z(G) = 1$, then
there exists a Galois extension of $\Q(x)$, regular over $\Q$, with 
Galois group isomorphic to $G$ and with $r$ branch points if and only if $\HHH_r^{in}(G)$
has a $\Q$-rational point. (This also holds if $\Q$ is replaced throughout by any field of characteristic $0$).
\end{enumerate}
\end{thm}

The space $\HHH_r^{in}(G)$ admits an action of ${\rm Gal}({\bar \Q}/\Q)$. Thus a $\Q$-rational point 
of $\HHH_r^{in}(G)$ must lie in an irreducible component which is defined over $\Q$.

Hurwitz spaces are covering spaces. In our situation, where $X/G \cong P^1\C$, the base space
of $\HHH_r^{in}(G)$ is the configuration space of $P^1\C$ with $r$ marked points. That is the space
$$D_r:=\{ S \subset P^1\C \ : \ |S| = r \}/\PGL_2(\C)$$
$$ = (P^1\C \setminus \{0,1,\infty\})^{r-3} \setminus \Delta_{r-3}.$$ Where 
$$\Delta_{r-3}:= \{(x_1,\dots, x_{r-3}) \in (P^1\C \setminus \{0,1,\infty\})^{r-3} \ : \ \exists i,j \ \mbox{with} \ x_i = x_j \}.$$
 The fundamental group of $D_r$
is the Hurwitz braid group on $r$-strings and is a quotient of the Artin braid group. 
Thus the connected components of $\HHH_r^{in}(G)$ are the orbits of the fundamental group of $D_r$ 
on the fibres. 
We define ${\mathcal T}_{r}(G)$ to be those elements $(\t_1,\dots,\t_r) \in G^r $
such that 
$$G = \langle (\t_1,\dots,\t_r) \rangle,$$  $$\prod_{i=1 }^r \t_i = 1,$$ and 
$$|G| (\sum_{i = 1}^r (|\t_i|-1)/|\t_i|) = 2(|G|+g -1).$$ 

The fibres of $\HHH_r^{in}(G)$ are parametrized by elements of 
${\underline \tau} \in {\mathcal F}_r(G) := {\mathcal T}_r(G)/G $ ,
where the action of $G$ on $G^r$, and hence on ${\mathcal T}_r(G)$, is via diagonal conjugation.
The action of the fundamental group of $D_r$ on ${\mathcal T}_r(G)$
is the well known action of the Artin braid group which we will define in the next section. This
action commutes with the action of $G$ via diagonal conjugation and 
hence induces a well defined action on ${\mathcal F}_r(G)$. 
From the definition of the action it is clear that the action of the Artin braid group preserves 
the set of conjugacy classes $C_i :=\t_i^G$ of elements of $\t_i \in \ut$. 
For an r-tuple of conjugacy classes $C_1,\dots,C_r$ we define the subset
$${\mathcal Ni}(C_1,\dots,C_r) :=\{\ut \in {\mathcal F}_r(G) \ : \ \exists
\s \in S_r \ \mbox{such that} \ \t_i \in C_{i\s} \ \mbox{for all} \ i \},$$
called the {\it Nielsen class} of $C_1,\dots,C_r$. The braid group action on ${\mathcal F}_r(G)$
preserves Nielsen classes which implies that connected components of $\HHH_r^{in}(G)$ are parametrized 
by braid orbits on Nielsen classes. The subset $\HHH(G,C_1,...,C_r)^{in} \subset \HHH(G)_r^{in}$ of 
$G$-curves $X$ with $g(X/G) = 0$ is a union of components parametrized by ${\mathcal Ni}(C_1,\dots,C_r) $.
By slight abuse of notation it is also called a Hurwitz space.

Generally it is very difficult to determine the number of braid orbits on Nielsen classes and hence not 
too much is known in general. There is the celebrated result of Clebsch, alluded to above, where he shows 
that if $G = S_n$ and all elements of $\ut$ are transpositions, then the corresponding Hurwitz space
$\HHH(G,C_1,...,C_r)^{in}$ is connected. His result was recently generalized by Liu and Osserman \cite{LO} 
who show that if $G = S_n$ and $C_i$ is represented by $g_i$ where each
$g_i$ is a single cycle of length $|g_i|$, then $\HHH(G,C_1,\dots,C_r)$ is connected.
 
On the other hand, Fried \cite{F} showed that if $G =A_n$, $g>0$ and all $C_i$ are represented 
by $3$-cycles, then $\HHH(G,g_0,C_1,\dots,C_r)$, the space of $G$ curves with $g_0 = g(X/G)$ and 
ramification in classes $C_1,\dots, C_r$, has one component if $g_0 = 0$ and two components if $g_0 > 0$.
In the latter case the components are separated by the lifting invariant. 

Finally we mention the theorem of Conway-Parker which shows that if the Schur multiplier of $G$ is generated
by commutators and the ramification involves all conjugacy classes of $G$ sufficiently often, then the 
corresponding Hurwitz space is connected, hence defined over $\Q$.

Nevertheless deciding whether or not $\HHH(G,g_0,C_1,\dots,C_r)$ is connected is still an open problem, both 
theoretically and algorithmically.
The algorithmic difficulties are due to the fact that the length of the Nielsen classes involved grows quickly.
 The package BRAID developed by Magaard, Shpectorov and V\"olklein
\cite{MSV} computes braid orbits algorithmically. This package is being upgraded by James, Magaard, Shpectorov \cite{JMS} to generalize to 
the situation of orbits of the mapping class group on the fibres of the 
Hurwitz space $\HHH(G,g_0,C_1,...,C_r)$ of $G$-curves $X$ with $g(X/G) = g_0$.

In this paper we introduce an algorithm which is designed to deal with long Nielsen classes.
Our idea is to represent a Nielsen class as union of direct products of shorter classes, thereby 
enabling us to enumerate orbits of magnitude $k^3$ where $k$ is an upper bound for what our 
standard BRAID algorithm can handle.

As an application of our algorithm we classify all braid orbits of Nielsen classes of primitive affine genus zero systems. 
That is to say that we find the connected components of $\HHH(G,C_1,...,C_r)$ of $G$-curves $X$, 
where $G$ is primitive and affine with translation subgroup $N$ and point stabilizer $H$, such that $g(X/H) = 0 = g(X/G)$.
Recall that $G$ is primitive if and only if $H$ acts irreducibly on the elementary abelian subgroup $N$ via conjugation. 
Equivalently this means that $G$ acts primitively
on the right cosets of $N$ via right multiplication and $N$ acts regularly on them. We compute 
that there are exactly $939$ braid orbits of primitive affine genus zero systems with $G'' \neq 0$.
The distribution in terms of degree and number of branch points is given in Table \ref{agz}. This 
completes the work of Neubauer on the affine case of the Guralnick-Thompson conjecture.

\begin{table} \label{agz}
\caption{Affine Primitive Genus Zero Systems: Number of Components}
\centering
\scalebox{0.75}{
\begin{tabular}{|l|c|c|c|c|c|c|c|r|}
\hline
Degree & \# Group& \# ramification&\#comp's & \#comp's  & \#comp's   & \#comp's  & \#comp's & \#comp's\\
      & isom. types  & types & r=3  &  r=4       & r = 5    &   r = 6  & r = 7& total  \\  \hline
8 & 2 & 50 & 29 & 17 & 9 &3 & 1& 59 \\ \hline
16 & 14 & 86 & 441 & 69 & 18 & 2 & - & 530 \\ \hline
32 & 1& 24 &169 &3 &- & -&- & 172 \\ \hline
64 & 14 &34 & 69 & - & - & - & - & 69 \\ \hline
9 & 4 & 26 & 14 & 10 & 2 & - & - & 26 \\ \hline
27 & 6 & 20 & 32 & 2 & - & - & - & 34 \\ \hline
81 & 4 & 5 & 6 &- &- &- &- & 6 \\ \hline
25 & 7 & 19 & 16 & 3 &- &- &- & 19 \\ \hline
125 & 1 & 2 & 8 &- &- &- &- & 8\\ \hline
49 & 1 & 2 & 6 &- &- &- &- & 6\\ \hline
121 & 1 & 2 & 10 &- &- &- &- & 10\\ \hline \hline
Totals & 55 & 270 & 800 & 104 & 29 & 5 & 1 & 939 \\ \hline 
\end{tabular}
}
\end{table}

Strictly speaking our new algorithm is not needed to settle the classification of braid orbits of Nielsen classes 
of primitive affine genus zero systems. However the problem is a good test case for our algorithm both 
as a debugging tool and as comparision for speed. Indeed our new algorithm shortens run times of BRAID 
by several orders of magnitude. 

The paper is organized as follows. In Section 2 we describe our algorithm and illustrate it with
an example that stresses the effectiveness. In Section 3 we discuss how we find 
all Hurwitz loci of affine primitive genus zero systems, displayed in tables at the end.

\section{The Algorithm}

We begin by recalling some basic definitions. 
The {\it Artin braid group} $B_r$ has the following presentation in terms of generators 
 $$\{Q_1,Q_2,\dots,Q_{r-1}\}$$ and relations 
$$Q_iQ_{i+1}Q_i=Q_{i+1}Q_iQ_{i+1};$$
$$Q_iQ_j=Q_jQ_i \quad \mbox{for}\quad |i-j|\geq 2.$$
The action of $B_r$ on $G^r$, or  {\it braid action} for short, is defined for all $i=1,2,\dots,r-1$ via:
$$Q_i: \quad (g_1,\dots,g_i,g_{i+1},\dots,g_r)\rightarrow (g_1,\dots,g_{i+1},g_{i+1}^{-1}g_ig_{i+1},\dots,g_r).$$
 
Evidently the braid action preserves the product $\prod_{i=1}^r g_i$ and the set of conjugacy classes
$\{C_1,\dots,C_r\}$ where $C_i:=g_i^G$. If the classes $C_i$ are pairwise distinct, then 
$B_r$ permutes the set $\{C_1,\dots,C_r\}$ like $S_r$ permutes the set $\{1,\dots,r\}$ where
$Q_i$ induces the permutation $(i,i+1)$. Thus we see that $B_r$ surjects naturally onto $S_r$ with 
kernel $B^{(r)}$, the {\it pure Artin braid group }.

We note that the group $B^{(r)}$ is generated by the elements 
\begin{align*}
Q_{ij}&:=Q_{j-1}\cdot\cdot\cdot Q_{i+1}Q_i^2Q_{i+1}^{-1}\cdot\cdot\cdot Q_{j-1}^{-1}\\
&=Q_i^{-1}\cdot\cdot\cdot Q_{j-2}^{-1}Q_{j-1}^2Q_{j-2}\cdot\cdot\cdot Q_i,
\end{align*} 
for $1\leq i<j\leq r$. 

If $P$ is a partition of $\{1,\dots,r\}$ with stabilizer $S_P \leq S_r$, then we denote by $B_P$ 
the inverse image of $S_P$ in $B_r$. The group $S_P$ is customarily 
called a parabolic subgroup of $S_r$ and thus we call $B_P$ a {\it parabolic subgroup}
of the Artin braid group. Now $B^r$ acts on ${\mathcal Ni}(C_1,\dots,C_r)$ permuting the 
classes $C_i$. Clearly the subgroup of $B^r$ which preserves the order of the conjugacy classes is a 
parabolic subgroup of $B^r$ and thus, from now on, we assume that the set of conjugacy classes 
$\{C_1,\dots,C_r\}$ is ordered in such a way that if $C_i = C_j$, then for all $i < k < j$ we have 
$C_k = C_i$. Let $P := P_1 \cup \dots \cup P_s$ be the partition of $\{1,\dots,r\}$ obtained by 
defining $x \sim y$ if $C_x = C_y$. The parabolic subgroup $B_P \leq B_r$ stabilizes the order 
of the classes in $\{C_1,\dots,C_r\}$. 

\begin{lem} \label{gens} 
If $\{C_1,\dots,C_r\}$ are ordered as above and $P$ is the corresponding partition, then the set 
${\mathcal Ni}^o(C_1,\dots,C_r) :=\{ \ut \in {\mathcal Ni}(C_1,\dots,C_r) \ : \ \t_i \in C_i  \mbox{ for all } i \}$ is an orbit of $B_P$.
\end{lem}

This means that the orbits of $B_P$ on ${\mathcal Ni}^o(C_1,\dots,C_r)$ determine the components 
of $\HHH^{in}(G,C_1,\dots,C_r)$. As $B_P$-orbits are shorter than the corresponding $B_r$-orbits 
by a factor of $[S_r:S_P]$, this is a significant advantage. 

For the record we note:

\begin{lem} 
The set of $Q_{ij}$'s such that $i$ and $j$ lie in different blocks of $P$
together with the $Q_i$'s such that $i$ and $i+1$ lie in the same block of $P$ is a set of generators of $B_P$.
\end{lem}

\subsection{Nodes}

As we noted in the introduction, Nielsen classes tend to be very large and thus we need to find ways 
to handle them effectively. Our algorithm achieves efficiency by interpreting tuples as elements 
of a Cartesian product. For this to be compatible with the action of $B_P$ on ${\mathcal Ni}^o(C_1,\dots,C_r)$, or equivalently, with the action of  $B_P \times G$ on 
$${\mathcal T}^o(C_1,\dots,C_r) :=\{(\t_1,\dots,\t_r) \in {\mathcal T}_r(G) \ : \  \t_i \in C_{i} \ \mbox{for all} \ i \},$$ we need to make some additional 
definitions. 
Let $ 1 < k < r$ and define 
$$L_k:= \langle Q_i \ : \ i \leq k-1 \rangle \cap B_P,$$
$$R_k:= \langle Q_i \ : \ i \geq k+1 \rangle \cap B_P.$$

Clearly $[L_k,R_k]=1$ and every $B_P$-orbit on ${\mathcal Ni}^o(C_1,\dots,C_r)$ is a union of 
$(L_k\times R_k)$-orbits. Equivalently every $B_P \times G$-orbit on ${\mathcal T}^o(C_1,\dots,C_r)$
is a union of $(L_k\times R_k \times G)$-orbits, which we call {\it nodes}. We refer to $k$ as the 
{\it level}. Typically we  choose $k$ to be close to $r/2$. If $(g_1,\dots,g_r)$ is a representative 
of a level $k$ node, then we split it into its {\it head} $(g_1,\dots,g_k)$ and its {\it tail} 
$(g_{k+1},\dots,g_r)$.  Since our package BRAID works with product $1$ tuples we will identify the 
head and the tail with the product $1$ tuples  $(g_1,\dots,g_k,x)$ and $(y,g_{k+1},\dots,g_r)$ 
respectively, where $y = \prod_{i=1}^kg_i$, $x = \prod_{i=k+1}^rg_i$. We note that $x = y^{-1}$ and 
that the actions of $L_k$ and $R_k$ centralize $x$ and $y$. Hence the conjugacy class $C_x:= x^G$ 
is an invariant of the node, which we call the {\it nodal type}. The following is clear. 

\begin{lem} 
For every node the heads of all tuples in the node form an orbit under $L_k \times G$. Similarly 
the tails form an $R_k \times G$-orbit. 
\end{lem}

We refer to the orbits above as the head (respectively, tail) orbit of the node. With the notation 
as above we see that the head orbit is of ramification type $(C_1,\dots, C_k,C_x)$ and the tail orbit 
is of ramification type $(C_y,C_{k+1}\dots, C_r)$. This observation allows us to determine all 
possible head and tail orbits independently, using BRAID.  Note that subgroups generated by the 
head or tail may be proper in $G$.

Accordingly, we give the following definitions. For a ramification type $\{C_1,\dots, C_r\}$, the 
partion $P$ as above, and conjugacy classes $C$ and $D$ of $G$, we define  
$${\mathcal L}_{k,C}:=\{(g_1,\dots,g_k,x) \ : \ g_i \in C_i, x \in C \ \mbox{and} \ 1 = (\prod_{i = 1}^k g_i)x \},$$
$${\mathcal R}_{k,D}:=\{(y,g_{k+1},\dots,g_r) \ : \ y \in D, g_i \in C_i \ \mbox{and} \ 1 = y(\prod_{i = 1}^k g_i) \}.$$

Note that $L_k \times G $ acts on ${\mathcal L}_{k,C}$ for all choices of $C$, and 
$R_k \times G $ acts on ${\mathcal R}_{k,D}$ for all choices of $D$. By slight abuse of terminology 
we call $(L_k \times G)$-orbits of $\cup_{C} {\mathcal L}_{k,C}$  {\it heads} and 
$(R_k \times G)$-orbits of $\cup_{D} {\mathcal R}_{k,D}$ {\it tails}. Clearly, for each node, its 
head orbit is among the heads and its tail orbit is among the tails. Furthermore the node is a subset 
of the Cartesian product of its head and its tail. We can now restate our task of finding nodes as 
follows. We need to find pairs of heads and tails which can correspond to nodes and then 
identify nodes within the Cartesian product of the head and the tail. 

The first of these tasks is achieved with the following definition. A head in ${\mathcal L}_{k,C}$ {\it 
matches} a tail in  ${\mathcal R}_{k,D}$ if $D = C^{-1} := \{x^{-1} \ : \ x \in C \}$. Since matching 
is specified entirely in terms of $C$ and $D$, we note that either every head in ${\mathcal L}_{k,C}$ 
matches every tail in  ${\mathcal R}_{k,D}$, or ${\mathcal L}_{k,C} \times {\mathcal R}_{k,D}$ contains 
no matching pairs. The head and tail of a node must necessarily be matching. Experiments show that most 
pairs of mathching heads and tails lead to nodes. So no further restrictions are necessary for our 
algorithm. 

Suppose now that ${\mathcal H} \subset {\mathcal L}_{k,C}$ and  ${\mathcal T} \subset {\mathcal R}_{k,D}$ 
match; i.e. $D = C^{-1}$. Our task now is to find all nodes in ${\mathcal H} \times {\mathcal T}$. 
There are several issues that we need to address. First of all, a pair of representatives 
$(g_1,\dots,g_k,x) \in {\mathcal H}$ and $(y,g_{k+1},\dots,g_r) \in {\mathcal R}$ can only give 
a representative of a node if $y = x^{-1}$. Therefore ${\mathcal H} \times {\mathcal T}$ is not a union 
of nodes; in fact most pairs of representative tuples do not work. We address this as follows. 

Let us select a particular element $x_0 \in C$. A natural choice for $x_0$ is, for example, the minimal 
element of $C$ with respect to the ordering defined in GAP \cite{GAP}. Let $y_0 = x_0^{-1}$. For 
${\mathcal H}$ and ${\mathcal T}$ as above, we define 
${\mathcal H}_0:= \{ (g_1,\dots,g_k,x) \in {\mathcal H} \ : \  x = x_0 \}$ and 
${\mathcal T}_0:= \{ (y,g_{k+1},\dots,g_r) \in {\mathcal T} \ : \  y = y_0 \}$.
We call ${\mathcal H}_0$ and ${\mathcal T}_0$ the {\it shadows} of ${\mathcal H}$ and ${\mathcal T}$.

\begin{lem} 
The shadows  ${\mathcal H}_0$ and  ${\mathcal T}_0$ are orbits for $L_k \times C_G(x_0)$ and $R_k \times C_G(x_0)$ respectively.
\end{lem}

Our first issue is now resolved as the representatives of ${\mathcal H}_0$ and  ${\mathcal T}_0$ 
automatically combine to give a product $1$ tuple. Furthermore for a node ${\mathcal N}$ of type $C$ 
we can similarly define the {\it shadow of ${\mathcal N}$} to be 
${\mathcal N}_0 := \{(g_1,\dots,g_r) \in {\mathcal N} \ : \  \prod_{i=k+1}^r g_i = x_0 \}$.

\begin{lem} The shadow  ${\mathcal N}_0$ is an orbit for $L_k \times R_k \times C_G(x_0)$ and furthermore it fully lies in ${\mathcal H}_0 \times {\mathcal T}_0$
where ${\mathcal H}_0$ and ${\mathcal T}_0$ are the shadows of the head and tail of ${\mathcal N}$.
\end{lem}

Thus we may work exclusively with shadows of heads, tails and nodes.

Our second issue is that combining representatives of matching head and tail shadows may not produce 
a tuple in ${\mathcal T}(C_1,\dots,C_r)$, because it may not generate $G$. We define {\it prenodes} 
as $L_k \times R_k \times G$-orbits on 
$$\{ \ut \in C_1 \times \dots \times C_r\ :\ \prod_{i=1}^r \tau_i = 1 \}.$$ 
Clearly every node is a prenode. Our terminology, head, tail, type and shadow, extends to prenodes in the obvious way.

\begin{lem} If ${\mathcal H}$ and ${\mathcal T}$ are matching heads and tails, then ${\mathcal H}_0 \times {\mathcal T}_0$ is a disjoint union 
of prenode shadows.
\end{lem}

So now our task is to identify all prenodes within ${\mathcal H}_0 \times {\mathcal T}_0$, that is to 
find a representative for each prenode. To achieve this, we work at the level of $L_k$- and $R_k$-orbits 
of ${\mathcal H}_0$ and ${\mathcal T}_0$ respectively. Let ${\mathcal O}_h \subset {\mathcal H}_0$ 
be an $L_k$-orbit and ${\mathcal O}_t \subset {\mathcal T}_0$ be and $R_k$-orbit. We define normalizers
$$ N_{C_G(x_0)}({\mathcal O}_h):= \{ c \in C_G(x_0) \ : \ \ut^c \in {\mathcal O}_h \ \mbox{for all} \ \ut \in {\mathcal O}_h \} $$
and 
$$ N_{C_G(x_0)}({\mathcal O}_t):= \{ c \in C_G(x_0) \ : \ \us^c \in {\mathcal O}_t \ \mbox{for all} \ \us \in {\mathcal O}_t \}.$$
Because the $G$-action commutes with that of $L_k$ and $R_k$, it suffices to check the conditions above 
for just a single $\ut \in {\mathcal O}_h$ and a single $\us \in {\mathcal O}_t$, respectively.

\begin{prop}  \label{verts} 
If ${\mathcal O}_h \subset {\mathcal H}_0$ is an $L_k$-orbit and ${\mathcal O}_t \subset {\mathcal T}_0$ 
is an $R_k$-orbit, then the prenode shadows in ${\mathcal H}_0 \times {\mathcal T}_0$ are in one-to-one 
correspondence with the double cosets 
$$ N_{C_G(x_0)}({\mathcal O}_h)\backslash C_G(x_0) / N_{C_G(x_0)}({\mathcal O}_t).$$
If $\{d_1,\dots,d_s\}$ is a set of double coset representatives, then a set of representatives for the prenodes can be chosen as  
$\{ (g_1,\dots,g_k,g_{k+1}^{d_i},\dots,g_r^{d_i}) \ : \ 1\leq i \leq s \}$, where 
$(g_1,\dots,g_k,x_0)$ and $(y_0,g_{k+1},\dots,g_r)$ are arbitrary representatives of 
${\mathcal O}_h$ and ${\mathcal O}_t$, respectively.
\end{prop}

\begin{proof} 
Let ${\mathcal X}$ be the set of $L_k$-orbits of ${\mathcal H}_0$ and ${\mathcal Y}$ be the set of $R_k$-orbits of ${\mathcal T}_0$.
Clearly $C_G(x_0)$ acts transitively on ${\mathcal X}$ and on ${\mathcal Y}$ with point stabilizers $N_{C_G(x_0)}({\mathcal O}_h)$ and 
$N_{C_G(x_0)}({\mathcal O}_t)$ respectively. Furthermore, the prenode shadows correspond to the $C_G(x_0)$-orbits on ${\mathcal X} \times {\mathcal Y}$.
The latter correspond to the double cosets as above. 
\end{proof}

So to construct all nodes we proceed as follows:\\

\noindent{\bf Algorithm: Find all level  $\mathbf{k}$ nodes} 
\begin{itemize}
\item Input: A group $G$, conjugacy classes $C_1,\dots, C_r$ and an integer $1 \leq k \leq r$.

\item For each type $C$: 

\begin{itemize}
\item Set $D:=C^{-1}$ and find all heads and tails by using BRAID \cite{MSV}.
\item From each head and tail select its shadow.
\item For each pair of head and tail shadows compute the normalizers and the double coset representatives as in Proposition \ref{verts}.
\item For each prenode check whether or not its representative generates $G$. Store the prenodes that pass this test as nodes.
\end{itemize}

\item Output all nodes. Nodes are sorted by their type, head, tail, and double coset representative.
\end{itemize}
\medskip
We close this subsection with the observation that the sum of the lengths of the nodes is computable at this stage. This means 
that we have calculated $|{\mathcal T}(C_1,\dots,C_r)|$ by a method different from that of Staszewski and V\"olklein \cite{SV}.

\subsection{Edges}

Our next step is to define a graph on our set of nodes whose connected components correspond to the braid orbits on the Nielsen class
${\mathcal Ni}(C_1,\dots, C_r)$.

\begin{defn}
Let $\Gamma_k(C_1,\dots,C_r)$ be the graph whose vertices are the level $k$ nodes of ${\mathcal Ni}(C_1,\dots, C_r)$.
We connect two nodes ${\mathcal N}_1$
and ${\mathcal N}_2$ by an edge if and only if there exists a tuple $\ut \in N_1$ and an element $Q \in B_P$
such that $\ut Q \in {\mathcal N}_2$.
\end{defn}

We remark that it is clear that the connected components of $\Gamma_k(C_1,\dots,C_r)$ are complete graphs and are 
in one-to-one correspondence with the braid orbits on the Nielsen class ${\mathcal Ni}(C_1,\dots, C_r)$. \\

Our algorithm for connecting vertices is as follows. Let $S$ be the set of generators of $B_P$ as in Lemma \ref{gens} minus those which are
contained in $L_k \times R_k$. For each node ${\mathcal N}$ we select a random tuple $\ut \in {\mathcal N}$ and apply a randomly
chosen generator $Q \in S$ to it. Using the head and tail of $\ut Q$ we find the node ${\mathcal N}'$ which contains it.
If ${\mathcal N} \neq {\mathcal N}'$ we record the edge. We repeat this until we have $s$ successes at ${\mathcal N}$.
Note that this does not mean that we find $s$ distinct neighbors for ${\mathcal N}$. If after a pre-specified number of tries $t$ we have 
no successes then we conclude that ${\mathcal N}$ is an isolated node; i.e. it is a $B_P$-orbit.

Once we have gone through all nodes, we obtain a subgraph $\Gamma'$ of  $\Gamma_k(C_1,\dots,C_r)$. We now find the connected components 
of $\Gamma'$ and claim that these are likely to be identical to those of  $\Gamma_k(C_1,\dots,C_r)$. Clearly if $\Gamma'$ is connected, 
then so is $\Gamma_k(C_1,\dots,C_r)$. Hence in this case our conclusion is deterministic. In other cases our algorithm is Monte-Carlo.

Based on our experiments, the situation where $\Gamma_k(C_1,\dots,C_r)$ is connected is the most likely outcome. It is interesting that 
even for small values of $s$ we tend to get that $\Gamma'$ is connected whenever $\Gamma_k(C_1,\dots,C_r)$ is connected. Also 
$t$ does not need to be large because if ${\mathcal N}$ is not isolated then almost any choice of $\ut$ and $Q$ will produce an 
edge. This, together with the way we represent tuples as products of heads and tails, makes this part of the algorithm very fast. 

Here is the formal describtion of the second part of the algorithm. \\

\noindent{\bf Algorithm: Finding the braid orbits} 
\begin{itemize}
\item Input: The $k$-nodes of ${\mathcal Ni}(C_1,\dots, C_r)$ arranged in terms their type, head, tail 
and double coset representative.

\item Initialize the edge set $E$ to the empty set.
\item For each node ${\mathcal N}$:

\begin{itemize}
\item Set counters c and d to 0.
\item Generate a random tuple $\ut$ from  ${\mathcal N}$ by selecting random head and tail.
\item Apply a randomly chosen $Q \in S$ to $\ut$.
\item Identify the node ${\mathcal N}'$ containing $\ut Q$ via its head and tail.
\item If ${\mathcal N} \neq {\mathcal N}'$, then 
\begin{itemize}
\item Set c to c+1 and set d to d+1.
\item Add the edge $({\mathcal N}, {\mathcal N}')$ to $E$ unless it is already known.
\end{itemize}
\item Else, 
\begin{itemize}
\item Set c to c+1.
\end{itemize}
\item Repeat this until either $d = s$ or $d = 0$ and $c = t$.
\end{itemize}

\item Determine and output the connected components of the graph $\Gamma'$ whose vertices are the input nodes and whose edge set is $E$.
\end{itemize}
\medskip

\subsection{Type $\mathbf{\{1_G\}}$ nodes}

During the development of the algorithm we noticed that a significant number of nodes are of type 
$C = \{ 1_G\}$, often more than half of all nodes. This can be explained by the fact that $C_G(1_G) = G$ is largest among all classes. Furthermore, all computations 
for these nodes are substantially slower than for nodes of types not equal to $C$. 
The next lemma gives a criterion when such nodes can be disregarded.

\begin{lem}\label{skip1}
Suppose ${\mathcal N}$ is a prenode of type $\{ 1_G\}$ and $\ut$ is its representative. 
Let $H$ and $T$ be the subgroups generated by the head and tail of $\ut$, respectively.
If $H$ and $T$ do not centralize each other, then the $B_P$-orbit containing ${\mathcal N}$
contains also a prenode ${\mathcal N}'$ of type not equal to $\{ 1_G\}$.
\end{lem}

\begin{proof}
Let $\ut=(g_1,\dots,g_r)$. Then $H=\langle g_1,\dots,g_k\rangle$ and $T=\langle g_{k+1},\dots,g_r\rangle$. 
We can also take a different set of generators for $H$, namely, the partial products 
$h_i=\prod_{j=i}^k g_j$, $i=1,\dots,k$. Since $H$ and $T$ do not centralize each other, 
some $h_s$ does not commute with some $g_t$, where $s\le k$ and $t>k$. It is now straightforward to see 
that that the pure braid $Q_{st}$ takes $\ut$ to a tuple, whose type is different from the type 
of $\ut$.
\end{proof}

This criterion is in fact exact. Indeed, it is clear that if $H$ and $T$ centralize each other then 
no pure braid (and more generally, no braid that preserves head and tail classes) can change the type. 
So the type within the $B_P$-orbit can change only if the same conjugacy class is present in the head 
and in the tail. However, a class that is present both in the head and in the tail must be central, 
and so the type still cannot change. Hence when $H$ and $T$ centralize each other then then 
the type (whether identity or not) remains constant on the entire $B_P$ orbit.

When the prenode ${\mathcal N}$ is a node, we have $G=\langle H,T\rangle$, and so the condition 
in the lemma fails very rarely. Thus, in most cases we need not consider nodes of type $\{1_G\}$. 
This turns out to be a significant computational advantage.

\subsection{An Example}
Let $G = {\rm AGL}_4(2)$, the group of affine linear transformations, acting on the 
16 points of $\F_2^4$; the vector space of dimension $4$ over the 
field of $2$ elements. $G$ has a unique conjugacy class of involutions whose elements have precisely 
$8$ fixed points in their action on $\F_2^4$ (we call this class $2A$) and another whose elements have 
exactly $4$ fixed points in their action on $\F_2^4$ (we call this second class $2B$).
We consider the ramification type ${\bar C}= (2A,2A,2A,2B,2B,2B)$. 
The structure constant for ${\bar C}$ is 
$21,267,671,040$; i.e. $$|{\mathcal T}(2A,2A,2A,2B,2B,2B)| \leq 21,267,671,040 .$$ This 
yields that an upper bound for the size of the corresponding $B_P$-orbit is $65,934$.
The available version of our package BRAID finds an orbit of this size within minutes. 
However, verifying that there is only one generating orbit takes days. This is 
due to the fact that BRAID spends most of its time searching for non-generating tuples in order to account 
for the full structure constant. Staszewski 
and V\"olklein \cite{SV} provided us with the function {\it NumberOfGeneratingNielsenTuples} which often 
helps to get around this problem. However, in this example 
the function runs out of memory on a $64G$ computer. On the other hand, 
after splitting ${\bar C}$ across the middle into $(2A,2A,2A,C)$ and $(D,2B,2B,2B)$ we compute
heads and tails within minutes.

\begin{table}[ht]
\caption{Time spent on generating heads and tails}
\centering
\begin{tabular}{c c c c}
\hline\hline
half & total number of orbits & time spent & type with the most orbits\\[0.5ex]
\hline
(2A,2A,2A,C) & 155 & 2 mins & (2A,2A,2A,2A)\\
(D,2B,2B,2B) & 619 & 10 mins & (4B,2B,2B,2B)\\ [1ex]
\hline
\end{tabular}
\end{table}

The step of contructing all nodes also takes little time. The group $G$ has 
$24$ non-identity conjugacy classes and hence we have $24$ types of nodes.
\begin{table}[ht]
\caption{Results from the function {\bf AllMatchingPairs}}
\centering
\begin{tabular}{c c c c c}
\hline\hline
number of total pairs &  most pairs & least pairs & number of types with no pairs \\[0.5ex]
\hline
903 & 3A & 4E & 11\\[1ex]
\hline
\end{tabular}
\end{table}


As shown in the table, our graph $\Gamma_k(2A,2A,2A,2B,2B,2B)$ has $903$ vertices. 
Drawing edges and checking that the graph $\Gamma'$ is connected took less than $5$ minutes.
The result is that $\Gamma_k(2A,2A,2A,2B,2B,2B)$ is connected, which means that 
the Hurwitz space $\HHH({\rm AGL}_4(2),2A,2A,2A,2B,2B,2B)$ is connected.

\section{Genus Zero Systems and the Guranlick-Thompson Conjecture}
We now come to our main application. We recall some background.
Suppose $X$ is a compact, connected Riemann surface of genus $g$,
and $\phi: X \rightarrow P^1\C$ is meromorphic of degree $n$. 
Let $B := \{ x \in P^1\C \ : \ |\phi^{-1}(x)| < n \}$ be the set of branch points of $\phi$.
It is well known that $B$ is a finite set and that if $b_0 \in P^1\C \setminus B$, then the fundamental group 
$\pi_1(P^1\C \setminus B,b_0)$ acts transitively on $F:= \phi^{-1}(b_0)$ via path lifting. 
The image of the action of $\pi_1(P^1\C \setminus B,b_0)$ on $F$ is 
called the {\it monodromy group} of $(X,\phi)$ and is denoted by $\Mon(X,\phi)$.
\\

We are interested in the structure of the monodromy group when the genus of $X$ is less than or equal to two and 
$\phi$ is indecomposable in the sense that there do not exits holomorphic functions $\phi_1 : X \rightarrow Y$ and $\phi_2 : Y \rightarrow P^1\C $
of degree less than the degree of $\phi$ such that $\phi = \phi_1 \circ \phi_2$. The condition that $X$ is connected implies that $\Mon(X,\phi)$ acts transitively on $F$, 
whereas the condition that $\phi$ is indecomposable implies that the action of $\Mon(X,\phi)$ on $F$ is primitive.

Our first question relates to a conjecture made by Guralnick and Thompson \cite{GT} in 1990. By $cf(G)$ we denote the set of 
isomorphism types of the composition factors of $G$. In their paper \cite{GT} Guralnick and Thompson defined the set 
$${\mathcal E}^*(g) = (\bigcup_{(X,\phi)} cf(\Mon(X,\phi))) \  \setminus \ \{A_n, \Z/p\Z \ : \ n > 4 \ , \ p \ \mbox{a prime} \}$$
 where $X \in \MMM(g)$, the moduli space 
of curves of genus $g$, and $\phi : X \longrightarrow P^1(\C)$ is meromorphic. 
They conjectured that ${\mathcal E}^*(g)$
 is finite 
for all $g \in \N$. Building on work of Guralnick-Thompson \cite{GT}, Neubauer \cite{N92}, Liebeck-Saxl \cite{LSa}, and Liebeck-Shalev
\cite{LSh}, the conjecture was established in 2001 by Frohardt and Magaard \cite{FM}.\\

The set ${\mathcal E}^*(0)$ is distinguished in that it is contained in ${\mathcal E}^*(g)$  for all $g$. Moreover 
the proof of the Guranlick-Thompson conjecture shows that it is possible to compute ${\mathcal E}^*(0)$ explicitly. 

The idea of the proof of the  Guranlick-Thompson conjecture is to employ Riemann's Existence Theorem to
translate the geometric problem to a problem in group theory as follows.
If $\phi: X \rightarrow P^1\C$  is as above with branch points $B = \{b_1,\dots, b_r\}$,
then the set of elements $\a_i \in \pi_1(P^1\C \setminus B,b_0)$, each represented by a simple loop around 
$b_i$, forms a canonical set of generators of $\pi_1(P^1\C \setminus B,b_0)$. Let $g$ is the genus 
of $X$. We denote by $\s_i$ the image $\a_i$ in $\Mon(X,\phi)\subset S_F \cong S_n$. Thus we 
have that $$\Mon(X,\phi) = \langle \s_1,\dots,\s_r\rangle \subset S_n $$ and that 
$$ \Pi_{i = 1}^r \s_i = 1.$$ Moreover the conjugacy class 
of $\s_i$ in $\Mon(X,\phi)$ is uniquely determined by $\phi$. Recall that the index of a permutation $\s \in S_n$ 
is equal to the minimal 
number of factors needed to express $\s$ as a product of transpositions.
The Riemann-Hurwitz formula asserts that 
$$ 2(n+g-1) = \sum_{i=1}^r \ind(\s_i). $$ 
 
\begin{defn} If $\t_1,\dots,\t_r \in S_n$ generate a transitive subgroup $G$ of $S_n$ 
such that 
$\Pi_{i=1}^r \tau_i = 1$ and  $ 2(n+g-1) = \sum_{i=1}^r \ind(\tau_i) $ for some $g \in \N$, then 
we call $(\tau_1,\dots,\tau_r)$ a {\it genus $g$ system} and $G$ a genus $g$ group. We call a genus $g$ system 
$(\tau_1,\dots,\tau_r)$ primitive if the subgroup of $S_n$ it generates is primitive. 
\end{defn}

If $X$ and $\phi$ are as above, then we say that 
$(\s_1,\dots,\s_r)$ is the genus $g$ system induced by $\phi$. 

\begin{thm}[Riemann Existence Theorem] 
For every genus $g$ system $(\tau_1,\dots,\tau_r)$ in $S_n$, there exists a Riemann surface $Y$ and a cover 
$\phi': Y \longrightarrow P^1\C$ with branch point set $B$ such that the genus $g$ system induced by 
$\phi'$ is $(\tau_1,\dots,\tau_r)$.
\end{thm}

\begin{defn}
Two covers $(Y_i,\phi_i)$, $i = 1,2$ are equivalent if there exist holomorphic maps $\xi_1:Y_1\longrightarrow Y_2$ 
and 
$\xi_2:Y_2\longrightarrow Y_1$ which are inverses of one another,
such that 
$\phi_1 = \xi_1\circ\phi_2$ and $\phi_2 = \xi_2\circ\phi_1$.
\end{defn}

The Artin braid group acts via automorphisms on $\pi_1(P^1\C \setminus B,b_0)$. We have that 
all sets of canonical generators of $\pi_1(P^1\C \setminus B,b_0)$ lie in the same braid orbit.
Also the group $G$ acts via diagonal conjugation on genus $g$ generating sets. The 
diagonal and braiding actions on genus $g$ generating sets commute and preserve equivalence 
of covers; that is, if two genus $g$ generating sets lie in the same orbit under either 
the braid or diagonal conjugation action, then the corresponding covers given by Riemann's Existence
Theorem are equivalent. We call two genus $g$ generating systems {\it braid equivalent} if they 
are in the same orbit under the group generated by the braid action and diagonal conjugation. 
We have the following result, see for example \cite{V}, Proposition 10.14.

\begin{thm}
Two covers are equivalent if and only if the corresponding genus $g$ systems 
are braid equivalent.
\end{thm}

Suppose now that $(\tau_1,\dots,\tau_r)$ is a primitive genus $g$ system of $S_n$. Express each 
$\tau_i$ as a
product of a minimal number of transpositions; i.e. $\tau_i := \prod_j \s_{i,j}$. The 
system $(\s_{1,1},\dots,\s_{r,s})$ is a primitive genus $g$ system generating $S_n$ consisting of 
precisely $2(n+g-1)$ transpositions. By a famous result of Clebsch, see Lemma 10.15 in \cite{V},
any two primitive genus $g$ systems of $S_n$ are braid equivalent. Thus we see that every genus $g$  
system can be obtained from one of $S_n$ which consists entirely of transpositions.

Thus, generically we expect primitive genus $g$ systems in $S_n$ to generate either $A_n$ or $S_n$.
We define $P{\mathcal E}^*(g)_{n,r}$ to be the braid equivalence classes of 
genus $g$ systems $(\tau_1,\dots,\tau_r)$ in $S_n$ such that $G:=\langle \tau_1,\dots,\tau_r\rangle$
is a primitive subgroup of $S_n$ with $A_n \cap G \neq A_n$. We also define $G{\mathcal E}^*(g)_{n,r}$
to be the conjugacy classes of primitive subgroups of $S_n$ which are generated by a member of $P{\mathcal E}^*(g)_{n,r}$.

We also define $$P{\mathcal E}^*(g) := \cup_{(n,r) \in \N^2} P{\mathcal E}^*(g)_{n,r},$$
and similarly $$G{\mathcal E}^*(g) := \cup_{(n,r) \in \N^2} G{\mathcal E}^*(g)_{n,r}.$$

We note that the composition factors of elements of $G{\mathcal E}^*(g)$ are elments of ${\mathcal E}^*(g)$.

While our ultimate goal is to determine $P{\mathcal E}^*(g)$ where $g \leq 2$, we focus here
on the case $g=0$. 

Our assumption that $G = \Mon(X,\phi)$ acts primitively on $F$ is a strong one and allows us to organize our 
analysis along the lines of the Aschbacher-O'Nan-Scott Theorem exactly as was done in the original paper of 
Guralnick and Thompson \cite{GT}. We recall the statement of the Aschbacher-O'Nan-Scott Theorem from \cite{GT} 

\begin{thm} Suppose $G$ is a finite group and $H$ is a maximal subgroup of $G$ such that 
$$\bigcap_{g \in G} H^g = 1.$$
Let $Q$ be a minimal normal subgroup of $G$, let $L$ be a minimal normal subgroup of $Q$, and 
let $\Delta = \{L=L_1, L_2, \dots, L_t\} $ be the set of $G$-conjugates of $L$. Then $G= HQ$
and precisely one of the following holds: 
\begin{itemize}
\item[(A)] $L$ is of prime order $p$.
\item[(B)] $F^*(G) = Q \times R$ where $Q \cong R$ and $H \cap Q = 1$.
\item[(C1)] $F^*(G) = Q$ is nonabelian, $H \cap Q = 1$.
\item[(C2)] $F^*(G) = Q$ is nonabelian, $H \cap Q \neq 1 = L \cap H$.
\item[(C3)] $F^*(G) = Q$ is nonabelian, $H \cap Q = H_1 \times \dots \times H_t$, \\
where $H_i = H \cap L_i \neq 1$, $1 \leq i \leq t.$
\end{itemize} 
\end{thm}

The members of $G{\mathcal E}^*(0)$ that arise in case (C2) were determined by Aschbacher \cite{A}. In all such 
examples $Q = A_5 \times A_5$.  Shih \cite{Shi} showed that no elements of $G{\mathcal E}^*(0)$ arise in case (B) and 
Guralnick and Thompson \cite{GT} showed the same in case (C1). Guralnick and 
Neubauer \cite{GN} showed that the elements of $G{\mathcal E}^*(0)$ arising in case (C3) all have $t \leq 5$. 
This was strengthened by Guralnick \cite{Gs} to $t \leq 4$ and the action of $L_i$ on the cosets of $H_i$ is a member of 
$G{\mathcal E}^*(0)$. In case (C3), where $L_i$ is of Lie type of rank one, all elements of $G{\mathcal E}^*(0)$ and 
$G{\mathcal E}^*(1)$ were determined 
by  Frohardt, Guralnick, and Magaard \cite{FGM1}, moreover they show that $t \leq 2$. In \cite{FGM2}  Frohardt, 
Guralnick, and Magaard showed that if
$t=1$, $L_i$ is classical and $L_i/H_i$ is a point action, then $n = [L_i:H_i] \leq 10,000$. That result together with 
the results of Aschbacher, Guralnick and Magaard \cite{AGM} show that if $t=1$ and $L_i$ is classical then 
$[L_i:H_i] \leq 10,000$. In \cite{GS} Guralnick and Shareshian show tha $G \in G{\mathcal E}^*(0)_{n,r} = \empty$ if $r \geq 9$. 
Moreover they show that if $G \in G{\mathcal E}^*(0)_{n,r}$ with $F^*(G)$ is alternating of degree $d < n$, then
and $r \geq 4$ unless $|B| = 5$ and $n = d(d-1)/2$.

So for the case where $F^*(G)$ is a direct product of nonabelian simple groups a complete picture of the elements of $G{\mathcal E}^*(0)$
is emerging.\\

In case (A) above, the affine case, we have that $F^*(G)$ is elementary abelian and it acts regularly on $F$.
Case (A) was first considered by Guralnick and Thompson \cite{GT}. Their results were then strengthened by Neubauer \cite{N92}.
After that, case (A) has not received much attention, which is in part due to its computational complexity.

The starting point for our investigations is Theorem 1.4 of Neubauer \cite{N92}.

\begin{thm} [Neubauer] If $F^*(G)$ is elementary abelian of order $p^e$ and $X = P^1\C$, then one of the following is true:
\begin{enumerate}
\item $G'' = 1$ and $1 \leq e \leq 2$
\item $p = 2$ and $2 \leq e \leq 8$,
\item $p = 3$ and $2 \leq e \leq 4$,
\item $p = 5$ and $ 2 \leq e  \leq 3$,
\item $p = 7$ or $11 $ and $e = 2$.

\end{enumerate} 
\end{thm}

The groups $G$ with  $G'' = 1$ and $1 \leq e \leq 2$ are Frobenius groups and are well understood.
Thus we concentrate on the affine groups of degrees 
$$\{8,16,32,64,128,256,9,27,81,25,125,49,121\}.$$
Our results are recorded in the tables below. These tables were calculated in several steps which 
we will now outline.\\

\noindent{\bf Algorithm: Enumerating Primitive Genus Zero Systems of Affine Type} 
\begin{itemize}

 \item Look up the primitive affine groups $G$ of degree $p^e$ using the GAP function \\
{\tt AllPrimitiveGroups(DegreeOperation, $p^e$)}.
\item For every group $G$, calculate conjugacy class representatives and permutation indices.
\item Using the function {\tt RestrictedPartions}, calculate all possible ramification types
satisfying the genus zero condition of the Riemann-Hurwitz formula.
\item Let $V = \F_p^e = F^*(G)$. For each conjugacy class representative $x$ calculate 
$\dim_{V}(x)$ and use Scott's Theorem to eliminate those types from the previous step which can not 
possibly act irreducibly on $V$; i.e. can not generate a primitive group.
\item Calculate the character table of $G$ and discard those types for which the class structure 
constant is zero.
\item For each of the remaining types of length four or more use the old version of BRAID, if possible, or
else run our new algorithm. For tuples of length three determine orbits via double cosets.
\end{itemize}

A few remarks are in order. First of all, the use of Scott's theorem above is best done in conjunction 
with a process called translation \cite{FM}. In fact, translation was crucial in handling certain types arising 
in degrees $128$ and $256$. Secondly, using BRAID on types of length $3$ is meaningless as every pure braid orbit 
has length one. Instead, we can compute possible generating triples using double cosets of centralizers.

\begin{appendix}


\begin{table}[ht]
\caption{The Genus Zero Systems for Affine Primitive Groups of Degree 8}
\centering
\scalebox{0.75}{
\begin{tabular}{|l|l|l|l|l|l|l|}
\hline
group& ramification & \# of & largest  & ramification & \# of & largest\\
      & type       &   orbits      & orbit    &   type   & orbits   & orbit\\ \hline
      
$A\Gamma L(1,8)$ & (3B,3B,6B) & 2 & 1   & (3A,3B,7B)   & 1 & 1   \\ \hline

                 & (3A,3B,7A) & 1 & 1   &(3A,3A,6A)& 2 & 1  \\ \hline
$ASL(3,2)$         & (4B,3A,7B) & 2 & 1   &(4B,3A,7A) & 2 & 1  \\
                 &(4B,3A,6A) & 4 & 1   & (4B,3A,4C) & 2 & 1   \\
                  &(4B,4B,7B) & 1 & 1   &(4B,4B,7A) & 1 & 1   \\
                  &(4B,4B,6A) & 2 & 1   &(4B,4B,4C) & 4 & 1   \\
                 &(2C,4B,7B) & 1 & 1   & (2C,4B,7A)& 1 & 1   \\
                 &(2B,7B,7B)& 1 & 1   & (2B,7A,7A)& 1 & 1   \\
                  &(2B,6A,7B)& 1 & 1   & (2B,6A,7A)& 1 & 1   \\
                  &(2B,3A,3A,3A)& 1 & 120   &(2B,4C,7B)& 1 & 1   \\
               &(2B,4C,7A)& 1 & 1   &(2B,4B,3A,3A)& 1 & 84   \\
                 &(2B,4B,4B,3A)& 1 & 66   &(2B,4B,4B,4B)& 1 & 36   \\
                &(2B,2C,3A,3A)& 1 & 30   &(2B,2C,4B,3A)& 1 & 24   \\
                &(2B,2C,4B,4B)& 1 & 24   &(2B,2B,3A,7B)& 1 & 21   \\
                  &(2B,2B,3A,7A)& 1 & 21   &(2B,2B,3A,6A)& 1 & 30   \\
                   &(2B,2B,3A,4C)& 1 & 24   &(2B,2B,4B,7B)& 1 & 14   \\
                  &(2B,2B,4B,7A)& 1 & 14   &(2B,2B,4B,6A)& 1 & 24   \\
                    &(2B,2B,4B,4C)& 1 & 24   &(2B,2B,2C,7B)& 1 & 7   \\
                 &(2B,2B,2C,7A)& 1 & 7   &(2B,2B,2B,3A,3A)& 1 & 864   \\
                 &(2B,2B,2B,4B,3A)& 1 & 648   &(2B,2B,2B,4B,4B)& 1 & 456   \\
                  &(2B,2B,2B,2C,3A)& 1 & 216   &(2B,2B,2B,2C,4B)& 1 & 192   \\
                   &(2B,2B,2B,2B,7B)& 1 & 147   &(2B,2B,2B,2B,7A)& 1 & 147   \\
                      &(2B,2B,2B,2B,6A)& 1 & 216   &(2B,2B,2B,2B,4C)& 1 & 192   \\
                   &(2B,2B,2B,2B,2B,3A)& 1 & 6480   &(2B,2B,2B,2B,2B,4B)& 1 & 4800   \\
                   &(2B,2B,2B,2B,2B,2C)& 1 & 1680   &(2B,2B,2B,2B,2B,2B,2B) & 1 & 48960   \\ \hline      

\end{tabular}
}
\end{table}

\begin{table}[ht]
\caption{The Genus Zero Systems for Primitive Groups of Degree 25 and 125}
\centering
\scalebox{0.8}{
\begin{tabular}{|l|l|l|l|l|l|l|}
\hline
group& ramification & \# of & largest  & ramification & \# of & largest\\
      & type       &   orbits      & orbit    &   type   & orbits   & orbit\\ \hline

$5^2:3$ &  (3B,3B,3B) & 8 & 1   &(3A,3A,3A)& 8 & 1  \\ \hline

$5^2:6$ &  (2A,3B,6B) &4 & 1   & (2A,3A,6B) & 8 & 1   \\ \hline

$5^2:S_3$  &(2A,3A,10D)& 1 & 1   & (2A,3A,10C) & 1 & 1   \\
           & (2A,3A,10B) & 1 & 1   &(2A,3A,10A) & 1 & 1   \\ \hline
           
$5^2:D(2*6)$  & (2A,2B,2C,3A)   & 1 & 12 & & &  \\ \hline

$5^2:D(2*4):2$  & (2A,2C,2D,4C)& 1 & 1   & (2A,2C,2D,4A)  & 1 & 1 \\ \hline

$5^2:O+(2, 5)$  &  (2C,4F,8A) & 1 & 1   &(2C,4E,8B) & 1 & 1   \\ \hline

$5^2:((Q_8:3)'2)$ &  (2B,3B,12B) & 1 & 1   &(2B,3B,12A)& 1 & 1   \\
                   &  (2B,3A,12D) & 1 & 1   &(2B,3A,12C)& 1 & 1   \\ \hline
 $5^2:((Q_8:3)'4)$  &   (4F,3A,4E)  & 1 & 1   &   (4D,3A,4G)  & 1 & 1   \\ \hline
 
$\ASL (2, 5):2$      & (2B,3A,20D) & 1 & 1   &(2B,3A,20C)& 1 & 1   \\
                    & (2B,3A,20B) & 1 & 1   & (2B,3A,20A) & 1 & 1   \\ \hline      
                    
$5^3:4^2:S_3$ &  (2B,3A,8B) & 4 & 1   &(2B,3A,8A)& 4 & 1  \\ \hline
      
\end{tabular}
}
\end{table}

\begin{table}[ht]\label{table2}
\caption{The genus zero system of $\AGL (4,2)$ Part 1}
\centering
\scalebox{0.8}{
\begin{tabular}{|l|l|l|l|l|l|}
\hline
\multicolumn{6}{|l|}{Using BRAID} \\ \hline
ramification & \# of & largest  & ramification & \# of & largest\\
 type       &   orbits      & orbit    &   type   & orbits   & orbit\\ \hline
(2B,5A,15B) & 1 &1 & (2B,2B,3B,7B) & 1 & 7\\
(2B,5A,15A) & 1 &1 &  (2B,2B,3B,7A) & 1 &7 \\
(2B,5A,14B) & 1 &1 & (2B,2B,4B,5A) &  1& 80\\
(2B,5A,14A) & 1 &1 & (2B,2B,4B,6C) & 1 & 96\\
(2B,6C,15B) & 1 &1 & (2B,2B,6A,5A) & 1 & 120\\
(2B,6C,15A) & 1 &1 & (2B,2B,6A,6C) & 1 & 108\\
(2B,6C,14B) & 1 &1 & (2B,2D,4D,5A) & 1 & 90\\
(2B,6C,14A) & 1 &1 & (2B,2D,4D,6C) & 1 & 78\\
(2D,4F,15B)& 1 &1 & (2B,2D,3A,5A) & 1 & 60\\
(2D,4F,15A)& 1 &1 & (2B,2D,3A,6C) & 1 & 72\\ \hline
(2D,6A,15B) & 3 &1 & (2B,2B,2B,2D,5A) & 1 & 650\\
(2D,6A,15A) & 3 &1 & (2B,2B,2B,2D,6C) & 1 & 648\\
(2D,6A,14B) & 2 &1 & (4B,4B,4D) & 12 & 1\\
(2D,6A,14A) & 2 &1 & (6A,4B,4D) &  18 & 1\\
(2B,2D,2D,15B) & 1& 15& (6A,4B,4B) & 32 & 1\\
(2B,2D,2D,15A) & 1& 15& (6A,6A,4F) & 12 &1 \\
(2B,2D,2D,14B) & 1& 14 & (6A,6A,4B) & 52 &1 \\
(2B,2D,2D,14A) & 1& 14 & (6A,6A,6A) & 72 & 1\\
(4D,4F,5A) & 6 &1 & (2B,2D,4B,4F) & 1 & 96\\
(4D,4F,6C) & 4 &1 &  (2B,2D,4B,4B) & 1 & 216\\ \hline
(4D,3B,7B) & 1 &1 & (2B,2D,6A,4F) & 1 & 84\\
(4D,3B,7A) & 1 &1 & (2B,2D,6A,4B) & 1 & 312\\
(4D,4B,5A) & 6 &1 & (2B,2D,6A,6A) & 1 & 414\\
(4D,4B,6C) & 12&1 & (2B,4F,4D,3B) & 1 & 24\\
(4D,6A,5A) & 18& 1& (2B,3A,4D,3B) & 1 & 30\\
(4D,6A,6C) & 12& 1& (2B,3A,3A,3B) & 1 & 24\\
(3A,4F,5A) & 2 & 1& (2D,2D,4D,4F) & 1 & 88\\
(3A,4F,6C) & 4 & 1 & (2D,2D,4D,4B) & 1  &192\\
(3A,6A,5A) & 10& 1&  (2D,2D,4D,6A) &  1 &336\\
(3A,6A,6C) & 12& 1&  (2D,2D,3A,4F) &  1 & 56\\
(2B,2B,4F,5A) & 1& 30& (2D,2D,3A,6A) &  1 & 216\\
(2B,2B,4F,6C) & 1& 30& (2B,2B,2B,4D,3B)  & 1 & 610\\ 
(2B,2B,2B,3A,3B) & 1& 216 &(2B,2B,2D,2D,4F) & 1& 576\\ \hline
\end{tabular}
}
\end{table}

\begin{table}[ht]\label{table3}
\caption{The genus zero system of $\AGL (4,2)$ Part 2}
\centering
\scalebox{0.8}{
\begin{tabular}{|l|l|l|l|l|l|l|l|}
\hline
\multicolumn{8}{|l|}{Using Matching Algorithm} \\ \hline
ramification & \# of & \# of  & orbit & ramification & \# of & \# of & orbit\\
 type  &  nodes  & orbits &  length&  type & nodes & orbits & length\\ \hline
(2B,2B,2D,2D,6A) & 170 & 1 & 2448 & (2B,2B,2B,2B,2B,3B) & 107 & 1& 1782 \\
(2B,2D,2D,2D,4D) & 63& 1 & 1920 & (2B,2B,2D,2D,4B)  & 151& 1 & 1920 \\
(2B,2B,2B,2D,2D,2D)& 903& 1 & 15168& (2B,2D,2D,2D,3A) & 56 & 1& 1512\\ \hline
\end{tabular}
}
\end{table}

\begin{table}[ht]\label{table4}
\caption{{\small Genus zero systems for other primitive affine groups of degree 16}}
\centering
\scalebox{0.7}{
\begin{tabular}{|l|l|l|l|l|l|l|l|}
\hline
Group & ramification & \# of & largest & Group & ramification & \# of & largest \\
      & type         & orbits & orbit   &      & type         & orbits & orbit  \\
\hline
$2^4:D(2*5)$ &  (2A,5B,4C)   & 1 & 1  &  $2^4.A_6 $ & (2B,5B,5B) & 4 & 1 \\
             &  (2A,5B,4B)   & 1 & 1  &              & (2B,5A,5A) & 4 & 1 \\
             &  (2A,5B,4A)   & 1 & 1  &              & (3A,3B,5B) & 2 & 1 \\
             &  (2A,5A,4C)   & 1 & 1  &              & (3A,3B,5A) & 2 & 1 \\
             &  (2A,5A,4B)   & 1 & 1  &              & (2B,2B,2B,5B)& 2 & 30 \\
             &  (2A,5A,4A)   & 1 & 1  &              & (2B,2B,2B,5A)& 2 & 30 \\
             &  (2A,2A,2A,4C)& 1 & 12 &              & (2B,2B,3A,3B)& 1 & 36 \\
             &  (2A,2A,2A,4B)& 1 & 12 &              & (2B,2B,2B,2B,2B)& 2 & 864 \\
             &  (2A,2A,2A,4A)& 1 & 12 &             &                 & & \\ \hline

  $(A_4\times A_4):2$  & (2A,6B,6C)   & 1 & 1 &       $2^4:S_5$ & (2C,5A,12A) & 1 & 1 \\
                           & (2A,6A,6D)   & 1 & 1 &       & (2C,5A,8A) & 1 & 1 \\
                             & (2A,2A,3E,3A)& 1 & 1 &        &(2E,6C,12A) & 1 & 1 \\
                              & (2A,2A,3D,3B)& 1 & 1   &    & (2E,6C,8A) & 1 & 1 \\ \cline{1-4}
   $(2^4:5).4 $  &(2A,4B,8B) & 1 & 1              &         & (2E,4E,12A)  & 1 & 1 \\
                  &(2A,4A,8A) & 1 & 1              &         &(2C,2E,2E,12A)& 1 & 6 \\ \cline{1-4}
$2^4:S_3\times S_3$  & (2E,6B,6C)  & 3 & 1   &               & (2C,2E,2E,8A)& 1 & 8 \\
                       & (2D,2E,2E,6C)& 1 & 12 &              &  (2D,6C,5A)& 3 & 1 \\
                         & (2C,2E,2E,6B)& 1 & 12 &           & (2D,4E,5A)& 3 & 1 \\
                      & (2C,2D,2E,6A)& 1 & 3  &               & (2C,2E,2D,5A)& 1 & 15 \\
                       & (2C,2D,2E,2E,2E)& 1& 48 &            & (2E,2E,2D,6C)& 1 & 18 \\ \cline{1-4}
$2^4.3^2:4 $            & (2C,4D,8B)& 2 & 1  &                & (2E,2E,2D,4E) & 1 & 24 \\
                    & (2C,4C,8A)& 2 & 1    &                   & (2C,2E,2E,2E,2D)& 1 & 120 \\
                       & (3A,4C,4D)& 3 & 1   &                 & & &                       \\ \hline

 $(S_4\times S_4):2$ & (2E,6B,8A)& 1 & 1       &               $2^4:A_5$ & (2C,5A,5B)& 3 & 1 \\
                      & (2C,4F,12A)& 1 & 1      &                        &(2C,6C,5B)& 1 & 1 \\
                       & (2C,6C,8A)& 1 & 1       &                        &(2C,6C,5A)& 1 & 1 \\
                       & (2E,2C,2D,8A)& 1 & 4    &                          &(2C,6B,5B)& 1 & 1 \\
                        & (2E,2C,2C,12A)& 1 & 2   &                     &(2C,6B,5A)& 1 & 1 \\
                        & (2F,4F,6B)& 3 & 1       &                  &(2C,6A,5B)& 1 & 1 \\
                         & (2E,2C,2F,6B)& 1 & 6    &                  &(2C,6A,5A)& 1 & 1 \\
                       & (2E,2C,3A,4F)& 1 & 6    &                    &(2C,2C,2C,5B)& 1 & 30 \\
                        & (2C,2D,2F,4F)& 1 & 12   &                &(2C,2C,2C,5A)& 1 & 30 \\
                       & (2C,2C,2F,6C)& 1 & 6    &                &(2C,2C,2C,6C)& 1 & 18 \\representatives for the prenodes.
                      & (2E,2E,2C,2C,3A)& 1 & 12 &                &(2C,2C,2C,6B)& 1 & 18 \\
                       & (2E,2C,2C,2D,2F)& 1 & 24 &            &(2C,2C,2C,6A)& 1 & 18 \\
                       & & & &                               &(2C,2C,2C,2C,2C)& 1 & 576 \\ \hline

$A\Gamma L(2,4) $   & (2C,4C,5A) & 1 & 1        &               $2^4.A_7$  & (2B,4A,14B)& 2 & 1 \\
                        &(2C,4C,15A) & 1 & 1        &                          & (2B,4A,14A)& 2 & 1 \\
                       &(3B,4C,6C) & 4 & 1          &                          & (2B,7B,6A)& 2 & 1 \\
                      & (2B,2C,3B,4C) & 1 & 20     &                            & (2B,7A,6A)& 2 & 1 \\  \cline{1-4}
$\ASL (2,4):2 $     & (2C,5A,6A)& 2 & 1                &                       &  (2B,5A,7B)& 2 & 1 \\
                      &(2B,6A,6A)& 2 & 1               &                         & (2B,5A,7A)& 2 & 1 \\
                     &(2B,2C,2C,5A)& 1 & 10            &                         &(3B,3A,7B)& 1 & 1 \\
                     &(2B,2B,2C,6A)& 1 & 12            &                    & (3B,3A,7A)& 1 & 1 \\
                   &(4A,4A,3A)& 2 & 1               &                        &(3B,4A,6A)& 6 & 1 \\
                   &(2B,2B,2B,2C,2C) & 1 & 80       &                       & (3B,4A,5A)& 10 & 1 \\  \cline{1-4}
$2^4.S_6 $        & (2B,2B,3B,5A) & 1 & 10          &                       &(2B,2B,2B,7B)& 2 & 21 \\
                    & (6B,4D,3B)   & 2 & 1            &                       &(2B,2B,2B,7A)& 2 & 21 \\
                  & (6B,6B,3B)   & 6 & 1           &                          &(4A,4A,4A)& 24 & 1 \\
                   & (2B,2D,4D,3B) & 1 & 12         &                           &(2B,2B,3B,4A)& 1 & 192 \\
                    & (2B,2D,6B,3B) & 1 & 24         &                          & & & \\
                    & (2B,2B,2D,2D,3B)& 1 & 108      &                        & & & \\ \hline

\end{tabular}
}
\end{table}

\begin{table}[ht]
\caption{The Genus Zero Systems for Affine Primitive Groups of Degree 32}
\centering
\scalebox{0.8}{
\begin{tabular}{|l|l|l|l|l|l|l|}
\hline
group & ramification & \# of & largest  & ramification & \# of & largest\\
       &type       &   orbits  & orbit  &   type   & orbits   & orbit\\ \hline

$\ASL (5,2)$&(2D,3B,31A) & 1 & 1     &     (2D,8A,6F) & 16 & 1 \\ 
&(2D,3B,31B) & 1 & 1     &      (2D,12A,6F) & 16 & 1 \\ 
&(2D,3B,31C) & 1 & 1     &     (2D,6E,6F) & 22 & 1 \\ 
&(2D,3B,31E) & 1 & 1     &     (2D,5A,6F) & 18 & 1 \\ 
&(2D,3B,31D) & 1 & 1     &      (4A,4A,6F) & 12 & 1 \\ 
&(2D,3B,31F) & 1 & 1     &      (4A,3B,8A) & 12 & 1 \\ 
&(2D,3B,30A) & 1 & 1     &      (4A,3B,12A) & 12 & 1 \\ 
&(2D,3B,30B) & 1 & 1     &     (4A,3B,6E) & 24 & 1 \\ 

&(2D,4J,21B) & 2 & 1     &      (4A,3B,5A) & 18 & 1 \\ 
&(2D,4J,21A) & 2 & 1     &      (4I,3B,4J) & 18 & 1 \\ 
 
&(6C,3B,4J) & 12 & 1     &     (2D,2D,2D,6F) & 1 & 720  \\ 

&(2B,2D,3B,4J) & 1 & 84   &    (2D,2D,4A,3B) & 1 & 624 \\ \hline

\end{tabular}
}
\end{table}

\begin{table}[ht]
\caption{The Genus Zero Systems for Affine Primitive Groups of Degree 64 }
\centering
\scalebox{0.8}{
\begin{tabular}{|l|l|l|l|l|l|l|}
\hline
 group & ramification & \# of & largest  & ramification & \# of & largest\\
      & type       &   orbits  & orbit  &   type   & orbits   & orbit\\ \hline

$2^6:3^2:S_3$ &  (2E,3F,12A) & 1 & 1 & (2E,6C,12B) & 1 & 1 \\ \hline
$2^6:7:6$ & (2E,3B,12B) & 1 & 1 & (2E,3A,12A) & 1 & 1 \\ \hline
$2^6:(3^2:3):D_8$ & (2G,4D,6D) & 3 & 1 & (2F,4D,6E) & 3 & 1 \\ \hline
$2^6:(3^2:3):SD_{16}$ &  (2E,4G,8D)  & 1 & 1 & (2E,4G,8B)  & 1 & 1 \\ \hline
$2^6:(6 \times GL(3, 2))$ & (2F,3C,14A) & 1 & 1 & (2F,3C,14B) & 1 & 1 \\ \hline
$2^6:S_7$ & (2I,4N,6K) & 4 & 1 & (2I,4D,7A) & 3 & 1 \\ \hline
$2^6:(\GL (2, 2) \wr S_3)$ & (2L,4N,6I) & 4 & 1 & & &  \\ \hline
$2^6:(\GL (3, 2) \wr 2)$ & (2J, 4Q, 14H)& 1 & 1 & (2J, 4Q, 14G)& 1 & 1 \\
                      & (2I, 2J, 2J, 7B)& 1 & 1 & (2I, 2J, 2J, 7A) & 1 & 1 \\ \hline
$2^6:7^2:S_3$ & (2C,3A,14C)& 1 & 1 & (2C,3A,14D)& 1 & 1 \\ 
                 & (2C,3A,14E)& 1 & 1 & (2C,3A,14F)& 1 & 1 \\ 
                 & (2C,3A,14G)& 1 & 1 & (2C,3A,14H)& 1 & 1 \\ \hline
$2^6:A_7$  & (2D,4F,7A)  & 2 & 1 & (2D,4F,7B) & 2 & 1 \\ \hline
$2^6:\GL (3, 2)$ & (2G,4F,8D)& 1 & 1 & (2G,4F,8B) & 1 & 1 \\
               & (2G,4D,6C)& 2 & 1 &     &    &  \\ \hline
$2^6:S_8$   &   (2C,6L,6K) & 4 & 1 & (2C,4O,7A)  & 6 & 1\\ \hline 

$2^6:GO-(6, 2)$ &    (4H,6C,12I)& 2 & 1 & (2C,8E,12I) & 6 & 1 \\ \hline  
$\AGL (6,2)$ & (2B,3B,15D)     & 4 & 1 & (2B,3B,15E) & 4 & 1 \\ \hline

\end{tabular}
}
\end{table}

\begin{table}[ht]
\caption{The Genus Zero Systems for Primitive Groups of Degree 9}
\centering
\scalebox{0.8}{
\begin{tabular}{|l|l|l|l|l|l|l|}
\hline
group& ramification & \# of & largest  & ramification & \# of & largest\\
      & type       &   orbits      & orbit    &   type   & orbits   & orbit\\ \hline
 $3^2:4$      &  (2A,4A,4A) & 2 &1 & (2A,4B,4B) & 2 & 1\\ \hline
 
 $3^2:D(2\times 4)$  & (2C,4A,6A) & 1 & 1   &    (2A,4A,6B) & 1 & 1 \\
                &  (2A,2C,2C,6A) & 1  & 2   & (2A,2A,2C,6B) & 1& 2 \\
                &  (2A,2C,2B,4A)  &  1 & 4 &      (2A,2A,2C,2C,2B) & 1& 8 \\ \hline
                
$3^2:(2'A_4)$   &   (3B,4A,3E) & 1 & 1   &(3B,6B,4A)  & 1 & 1 \\
                 &    (3B,6A,3D) & 1 & 1   &(3A,4A,3D)   & 1 & 1 \\
                 & (3A,6B,3E) & 1 & 1   &(3A,6A,4A)   & 1 & 1 \\
                 & (3B,3B,3B,2A) & 1 & 1   &(3A,3A,3A,2A)  & 1 & 1 \\ \hline
$A\Gamma L(1, 9) $  & (2A,4A,8A) & 1 & 1  &(2A,4A,8B)  & 1 & 1 \\ \hline

$\AGL (2, 3)$ &(2A,3C,8B)& 1 & 1   & (2A,3C,8A) & 1 & 1   \\
            &(2A,6A,8B) & 1 & 1   & (2A,6A,8A)  & 1 & 1  \\
            &(2A,2A,2A,8B) & 1 & 16   &(2A,2A,2A,8A) & 1 & 16\\
            &(2A,2A,3A,3C) & 1 & 12   &(2A,2A,3A,4A) & 1 & 12   \\
            &(2A,2A,3A,6A)& 1 & 12   & (2A,2A,2A,2A,3A)& 1 & 216\\ \hline
\end{tabular}
}
\end{table}

\begin{table}[ht]
\caption{The Genus Zero Systems for Primitive Groups of Degree 27}
\centering
\scalebox{0.8}{
\begin{tabular}{|l|l|l|l|l|l|l|}
\hline
group& ramification & \# of & largest  & ramification & \# of & largest\\
      & type       &   orbits      & orbit    &   type   & orbits   & orbit\\ \hline

$3^3.A_4$ & (2A,3B,9D)  & 1 & 1   &(2A,3B,9B)& 1 & 1   \\
           & (2A,3A,9C) & 1 & 1   & (2A,3A,9A)& 1 & 1   \\ \hline
           
$3^3(A_4\times 2)$  & (2B,3D,12B) & 1 & 1   & (2B,3D,12A) & 1 & 1 \\
                     & (2A,2B,2B,3D) &1 &24    &  &  &\\ \hline

  $ 3^3.S_4 $     &   (2B,4A,9B)  & 1 & 1   &(2B,4A,9A)  & 1 & 1   \\ \hline
 
$3^3(S_4\times 2)$    & (2E,4A,6G)& 4 & 1   & (2B,2E,2E,4A) & 1 & 16   \\ \hline

$\ASL (3,3)$          &      (2A,3F,13D)  & 2 & 1   &(2A,3F,13C)  & 2 & 1  \\
                     &    (2A,3F,13B)  & 2 & 1   &(2A, 3F,13A)  & 2 & 1   \\ \hline
 $ \AGL (3, 3)$        &(2C,4A,13D) & 1 & 1   &(2C,4A,13C) & 1 & 1   \\
                     &   (2C,4A,13B)& 1 & 1   & (2C,4A,13A)& 1 & 1   \\
                     & (3E,6E,4A) & 8 & 1      & & &  \\ \hline        
 
\end{tabular}
}
\end{table}

\begin{table}[ht]
\caption{The Genus Zero Systems for Primitive Groups of Degree 49}
\centering
\scalebox{0.8}{
\begin{tabular}{|l|l|l|l|l|l|l|}
\hline
group& ramification & \# of & largest  & ramification & \# of & largest\\
      & type       &   orbits      & orbit    &   type   & orbits   & orbit\\ \hline
      
$7^2:4$ & (2A,4B,4B) & 12 & 1   & (2A,4A,4A)   & 12 & 1   \\ \hline

$7^2:3:D(2*4)$ & (2A,4A,6C) & 3 & 1   &(2A,4A,6B)& 3 & 1  \\ \hline

\end{tabular}
}
\end{table}

\begin{table}[ht]
\caption{The Genus Zero Systems for Primitive Groups of Degree 81}
\centering
\scalebox{0.8}{
\begin{tabular}{|l|l|l|l|l|l|l|}
\hline
group& ramification & \# of & largest  & ramification & \# of & largest\\
      & type       &   orbits      & orbit    &   type   & orbits   & orbit\\ \hline
      
$3^4:(\GL (1, 3) \wr S_4)$ & (6S,4C,6K) & 2 & 1 & & & \\ \hline

$3^4:(2 \times S_5)$ &  (6K,4A,6M)  & 1 & 1 & & & \\ \hline

$3^4:S_5$ & (6E,12A,3G) & 1 & 1 & & & \\ \hline
                                         
$\AGL (4,3)$ & (2C,5A,8E) & 1 & 1 & (2C,5A,8F) & 1 & 1 \\ \hline

\end{tabular}
}
\end{table}

\begin{table}[ht]
\caption{The Genus Zero Systems for Primitive Groups of Degree 121}
\centering
\scalebox{0.8}{
\begin{tabular}{|l|l|l|l|l|l|l|}
\hline
group& ramification & \# of & largest  & ramification & \# of & largest\\
      & type       &   orbits      & orbit    &   type   & orbits   & orbit\\ \hline
      
$11^2:3$ & (3A,3A,3A) & 40 & 1   & (3B,3B,3B)   & 40 & 1   \\ \hline

$11^2:4$ & (2A,4A,4A) & 30 & 1   &(2A,4B,4B)& 30 & 1  \\ \hline

$11^2:6$ & (2A,3B,6B) & 20 & 1  & (2A,3A,6A)& 30 & 1 \\ \hline

$11^2:(Q_8:D_6)$ & (2B,3A,8A)  & 5 & 1  &      (2B,3A,8B)& 5 & 1  \\ \hline

\end{tabular}
}
\end{table}

\end{appendix}

\end{document}